\documentclass[oneside,reqno]{amsart}
\usepackage{amsmath}
  \usepackage{paralist}
  \usepackage{amssymb} 
  \usepackage{graphics} 
 \usepackage{graphicx,psfrag}
 \usepackage{epsfig} 
 \usepackage{float}
 \usepackage[colorlinks=true]{hyperref}

\hypersetup{urlcolor=blue, citecolor=red}

\numberwithin{equation}{section}

\newcounter{mnote}
 \newcommand{\mnote}[1]{\addtocounter{mnote}{ 1}
   \ensuremath{{}^{\bullet\arabic{mnote}}}
   \marginpar{\footnotesize\em\color{red}\ensuremath{\bullet\arabic{mnote}}\ #1}}

\newtheorem{theorem}{Theorem}[section]
\newtheorem{corollary}[theorem]{Corollary}

\newtheorem{proposition}[theorem]{Proposition}

\newtheorem{remark}{Remark}[section]



\def\cA{{\mathcal A}}

\def\AAA{ \mathcal{ A } }\def\BBB{ \mathcal{ B } }

\def\MMM{ \mathcal{ M } }\def\OOO{ \mathcal{ O } }

\def\calV{\mathcal{V}}

\def\RR{ {\mathbb R} }

\def\tc{\tilde{c}}

\def\cP{{\mathcal P}}

\def\bR{{\Bbb R}}

\def\k0{\kappa_0}

\def\leaderfill{\leaders\hbox to 1em{\hss-\hss}\hfill\ }

\def\mF{g}

\begin{document}
\title[1D parametric determining form for the 2D NSE]
{
One-dimensional parametric determining form for the two-dimensional Navier-Stokes equations}
\author{C. Foias$^{1}$}\address{$^1$Department of Mathematics\\
Texas A\&M University\\ College Station, TX 77843}
\author{M. S. Jolly$^2$}
\address{$^2$Department of Mathematics\\
Indiana University \\ Bloomington, IN 47405}
\author{D. Lithio$^{2,3}$}
\address{$^3$Allstate\\
 Northbrook, IL 60062}
\author{E. S. Titi$^{1,4}$}
\address{$^4$The Department of Computer Science and Applied Mathematics\\
The Weizmann Institute of Science, Rehovot 76100, Israel.}
\email[M. S. Jolly]{msjolly@indiana.edu}
\email[D. Lithio]{dlithio@gmail.com}
\email[Edriss S. Titi]{titi@math.tamu.edu and edriss.titi@weizmann.ac.il}


\date{\today}

\subjclass[2000]{35Q30, 76F02}
\keywords{Navier-Stokes equations, global attractors, determining nodes, determining form,
parametric determining form, determining parameter}

\begin{abstract}
The evolution of a determining form 
for the 2D Navier-Stokes equations (NSE), which is an ODE on a space of trajectories
is completely described.
It is proved that at every stage of its evolution, the solution
is a convex combination of the initial trajectory and the fixed
steady state, with a dynamical convexity parameter $\theta$, which will be called the characteristic determining parameter.  That is, we show a remarkable separation of variables  formula for the solution of the determining form. Moreover, for a given initial trajectory, the dynamics of the infinite-dimensional determining form are equivalent to those
of the characteristic determining parameter $\theta$ which is governed by a one-dimensional ODE.
This one-dimensional ODE is used to show that if the solution to the determining form converges to the fixed state it does so no faster than $\OOO(\tau^{-1/2})$,
otherwise it converges to a projection of some other trajectory in the global attractor of the NSE, but no faster than
$\OOO(\tau^{-1})$, as $\tau \to \infty$, where $\tau$ is the evolutionary variable in determining form.
The one-dimensional ODE also exploited in computations which suggest that the one-sided convergence rate estimates are in fact achieved.  The ODE is then modified to
accelerate the convergence to an exponential rate.
Remarkably, it is shown that the zeros of the scalar function that governs the dynamics of $\theta$, which are called
characteristic determining values, identify in a unique fashion the trajectories in the
global attractor of the 2D NSE.  Furthermore, the one-dimensional characteristic determining form enables us to find unanticipated geometric features of
the global attractor, a subject of future research.
 \end{abstract}
\maketitle
\centerline{\it This paper is dedicated to the memory of Professor George Sell, an initiator of modern approach}  
\centerline{\it to ODEs and infinite-dimensional dynamical systems theory, a great friend, collaborator and teacher.}

\section{Introduction}\label{sec1}

Many dissipative infinite-dimensional dynamical systems have been reduced
to finite-dimensional ordinary differential equations (ODE) by the restriction to
inertial manifolds \cite{CFNTbook, CFNT,inertial1,FSTi}.  The list includes
the Kuramoto-Sivashinsky, complex Ginzburg-Landau, and certain reaction diffusion equations \cite{M-PS,T97},
just to name a few.  Whether the 2D Navier-Stokes equations (NSE) enjoys a
finite-dimensional reduction via an inertial manifold
has remained an open question since the
mid-1980's.  Recently however, it has been shown that the global attractor $\cA$ of the 2D NSE
can in fact be captured by an ODE $dv/d\tau=F(v)$ in a Banach space of trajectories.  This is an ODE in the
true sense that $F$ is a globally Lipschitz map.  This ODE has its own
evolution variable denoted by $\tau$, which is distinct from that in the NSE, which we denote by $s$.
There have been two different constructions of such an ODE, referred to as a {\it determining form}
due its connection to the notion of determining modes, nodes, volume elements, etc.  \cite{CJTi1,FP,FTfinelt,FT1, FTi,HTi,JT93}.
In the first approach, in \cite{FJKrT1},  trajectories in the global attractor of the NSE,
are precisely traveling waves in the variables $\tau$ and $s$, while in the second, in \cite{FJKrT2},
they are precisely steady states of the determining form.

This paper focuses on the latter type
of determining form. We show in section \ref{sec3} that its evolution always proceeds along
a segment in $X$ the phase space of trajectories (see \eqref{X}), connecting the initial trajectory $v_0$, and $Ju^*$, where
$u^*$ is a fixed steady state $u^*$ of the NSE, and $J$ is a finite-rank projector
(see \eqref{ape1}, \eqref{ape2} below).
This leads to a separation of variables formula for the solution of the determining form, specifically, we show in section \ref{sec3}, that the solution is a convex combination $v(\tau,s)=\theta(\tau) v_0(s)+(1-\theta(\tau))Ju^*$, $s \in \bR$, $\tau \ge 0$.  We will refer to the convexity parameter
as the {\it characteristic determining parameter}.  Moreover, if
$v(\cdot) \neq Ju(\cdot)$ for any trajectory $u(\cdot)\subset \AAA$ and there is no
trajectory in $\cA$ strictly between $v_0$ and $Ju^*$, then $v(\tau) \to Ju^*$, as $\tau \to \infty$, no faster than $\OOO(\tau^{-1/2})$.  Otherwise $v(\tau) \to J\tilde{u}$  no faster than $\OOO(\tau^{-1})$, where $\tilde{u}(\cdot) \subset \cA\setminus \{u^*\}$ is on the
closed segment connecting $Ju^*$ with $v_0$.    For a given $v_0$, the evolution of the determining form is equivalent to the dynamics of a one-dimensional ODE in the characteristic determining parameter $\theta$,
$d \theta/d\tau=\Phi( \theta;v_0,u^*)$, which we refer to as the {\it characteristic parametric determining form}, where  $\Phi$ is a Lipschitz function in $\theta \in [0,1]$.
We show in section \ref{sec4} that the function $\Phi$ is readily modified to accelerate the convergence to an exponential rate.

This surprising reduction to a one-dimensional ODE may have far reaching consequences.
 It enables us to find unanticipated geometric features of the global attractor, and develop
alternative computational approaches to finding this set, a subject of ongoing research; in particular bifurcation analysis of the 2D NSE.  Remarkably, the zeros of the real function $\Phi$, which we refer to as {\it characteristic determining values}, identify the trajectories on the global attractor of the 2D NSE in a unique fashion.
We demonstrate the utility of
this reduction by finding a limit cycle to roughly machine accuracy with just eight steps
of the secant method.

In section \ref{sec6} we exploit this scalar ODE in computations which suggest that the
one-sided estimates on the two rates of convergence are in fact achieved.
A significant advantage of this approach 
is that it avoids the padding
of a time interval to account for compounding relaxation times. 
In contrast, the direct numerical simulation of the determining form $dv/d\tau=F(v)$ would involve the sequential evaluation of a map $W(v)$ (see Theorem \ref{mainthm}).  
The image of $W$ is the unique bounded solution $w(s)$ of a system similar to the 2D NSE (see \eqref{weqn}) but
driven by $v(\tau,s)$, where $\tau$ is fixed, and  $-\infty < s  < \infty$.
On a computer, trajectories over all time 
must be truncated to, say $0 \le s \le s_2$.  Given  
$v(\tau,s)$ for $s \in [0,s_2]$, the image $W(v)$ can be effectively
approximated over subinterval $s_1 \le s \le s_2$, after a short relaxation time $s_1 > 0$,
by solving this driven NSE system, with initial condition $w(0)=0$.  Thus in order to make $N$ sequential
evaluations of $F$, one would need to pad the initial trajectory $v(0,s)$, i.e., specify
$s$ over the interval $0 \le s \le Ns_1+s_2$.   In the case of the one-dimensional ODE, however,
$W(v)$ is always evaluated at a convex combination $v=\theta v_0+(1-\theta)Ju^*$, so the
relaxation times are not compounded. 

Notably, the 2D NSE considered in the paper is only a prototype example of a dissipative evolutionary equation; the results presented in this
work can be equally extended to other dissipative evolution equations
(cf. \cite{JMST,JST1,JST2}).

\section{Background and preliminaries} \label{sec2}
The two-dimensional incompressible Navier-Stokes equations (NSE)
\begin{equation}\label{NSES}
\begin{aligned}
&\frac{\partial u}{\partial s} -
\nu \Delta u  + (u\cdot\nabla)u + \nabla p = \mF \\
&\text {div} u = 0 \\
& \int_{\Omega} u\,  dx =0 \;,\qquad \int_{\Omega} \mF\,  dx =0 \\
& u(0,x) = u_0(x),
\end{aligned}
\end{equation}
subject to periodic boundary conditions with the basic periodic-domain $\Omega =[0,L]^2$, can be written as
an evolution equation in a Hilbert space $H$   (cf.  \cite{CF88,Robinson,T97})
%
\begin{equation}\label{NSE}
\begin{aligned}
&\frac{d}{ds}u(s) + \nu Au(s) + B(u(s),u(s)) = f, \,\, \text{for}\, s >0,\\
& u(0)=u_0.
\end{aligned}
\end{equation}
Here $H$ is the closure of $\calV$ in  $(L^2_{\rm{per}}(\Omega))^2$,  where
$$
\calV=\{\varphi:\, \varphi\,\, \text{is}\,\, \mathbb{R}^2-{\hbox{valued trigonometric polynomials}}, \,\, \nabla \cdot \varphi = 0, \, \text{ and} \, \int_{\Omega} \varphi(x)\, dx =0  \}\;.
$$
The Stokes operator $A$, the bilinear operator $B$, and force $f$ are defined as
\begin{equation}
\label{nabla1}
A=-\cP\Delta=-\Delta \;, \quad  B(u,v)=\cP\left( (u \cdot \nabla) v \right)\;,\quad f=\cP \mF\;,
\end{equation}
where $\cP$ is the Helmholtz-Leray orthogonal projector from $(L^2(\Omega))^2$ onto $H$.
We denote $|\cdot|=\|\cdot\|_{L^2}$ and $\|u\|=|A^{1/2}u|$, with
fractional powers of $A$ defined by $A^\alpha \varphi_j=\lambda_j^\alpha \varphi_j$, $\alpha \in \bR$, where
$\{\varphi_j\}$ is an othornormal basis for $H$ consisting of eigenfunctions of $A$ corresponding
to eigenvalues $\{\lambda_j\}$ satisfying
$$
0<\lambda _1=\left(\frac{2\pi}{ L}\right)^2
\leq \lambda _2\leq\lambda_3 \le \cdots \;.
$$

It is well known that for a time independent forcing term $f\in H$, the two-dimensional NSE \eqref{NSE} is globally well-posed, and that it has a global attractor
\begin{equation}\label{attractordef}
\mathcal{A}=\{u_0\in H: \, \exists \, \text{ a solution}\, u(s,u_0)\, \text{of \eqref{NSE}} \, \, \forall \ s\in\mathbb{R},\,\, \sup_{s\in \bR} \|u(s)\|<\infty\}\;.
\end{equation}
Moreover, it is also well known that
\begin{equation}\label{Grashof}
\mathcal{A} \subset \{u\in D(A^{1/2}): \|u\|\leq  G\nu\kappa_0\},
\end{equation}
where $\k0=\lambda_1^{1/2}=2\pi/L$ and $G=|f|/(\nu^2\lambda_1)$ is the Grashof
number, a dimensionless parameter which plays the role of the Reynolds number in turbulent flows \cite{CF88,FMRT,Hale,Robinson,T97}.
It is also proved in \cite{FJLRYZ} that
$$
|Au| \le c\nu\k0^2G^3\;,\quad  |A \frac{du}{ds}| \le c\nu^2\k0^4G^7 \quad \forall \ u \in \AAA \;.
$$

 For a subset $\MMM \subset (L^2_{\text{per}}(\Omega))^2$, we denote $\dot \MMM=\{\varphi \in \MMM:  \int_\Omega \varphi(x) \ dx =0\}$.
 Let $J=J_h$ be a finite-rank linear interpolant operator,
 $$
J:(\dot H_{\text{per}}^2(\Omega))^2 \to (\dot H_{\text{per}}^1(\Omega))^2\;,
$$  which approximates the identity in the
 sense that for every $w \in (\dot {H}^2(\Omega))^2$
\begin{align}
|Jw-w| &\le c_1h|\nabla w|+c_2h^2|\Delta w| \label{ape1} \\
|\nabla(Jw-w)|&\le \tilde{c}_1|\nabla w|+\tilde{c}_2h|\Delta w|  \label{ape2}\;,
\end{align}
where $h$ represents the spatial resolution of the interpolant operator $J$. Observe that the above approximate identity inequalities imply that $J:(\dot H^2)^2 \to(\dot H^1)^2$ is a bounded linear operator. Indeed, we have from \eqref{ape1} that for every $w \in (\dot{H}^2)^2$
$$
|Jw| \le |w| + c_1h|\nabla w| + c_2h^2|\Delta w| \sim \|w\|_{H^2}\;,
$$
while from \eqref{ape2}
$$
|\nabla Jw| \le |\nabla w| + \tc_1|\nabla w| + \tc_2h|\Delta w| \sim \|w\|_{H^2}\;.
$$
Adding these two inequalities implies
$$
\|Jw\|_{\dot H^1} \le c_J \|w\|_{\dot H^2} \;, \quad \text{for every} \ w \in (\dot{H}^2)^2\;.
$$

Such an interpolant operator could be defined in terms of Fourier modes, nodal values, volume elements,
or finite elements (see, e.g.,  \cite{AOT,CJTi1,FP,FTfinelt,FTi,HTi, JT93} and references therein).
Let $X$ denote the Banach space $C_b^1(\mathbb{R},J(\dot{H}^2_{\rm{per}}(\Omega))^2)$.   We endow $X$ with the norm
\begin{equation}\label{X}
\|v\|_X=\|v\|_{X^0}+\sup_{s\in \bR}\frac{ \|v'(s)\|}{\nu^2\kappa_0^3}\;,
 \quad \text{where}\quad
 \|v\|_{X^0}=\sup_{s \in \bR} \frac{\|v(s)\|}{\nu\kappa_0} \;.
  \end{equation}

Now consider for any $v\in X$
the associated ``feedback control" system
\begin{equation} \label{weqn}
\frac{dw}{ds}+Aw+B(w,w)=f-\mu\nu\kappa_0^2 (J_hw-v)\;, \quad w \in H\;.
\end{equation}
The following is proved in \cite{FJKrT2}.
\begin{theorem}\label{mainthm} \hfill\break
\begin{enumerate}
\item Let $\rho > 0$ and $\mu\gtrsim \rho^2$, $h \lesssim 1/\sqrt{\mu}$.  Then for every $v \in \BBB_X^\rho(0)$ equation \eqref{weqn} has a unique solution $w \in Y=C_b^1(\bR,D(A))$, defining a globally Lipschitz map
$W:\BBB_X^\rho(0)\to Y$, by $W(v)=w$,
\item $\| J u\|_X \le R \sim G^7$ for all  $u(\cdot) \subset \AAA$,
\item let $v \in \BBB^{4R}_X(0)$,  then $JW(v)=v$ if and only if $W(v)(\cdot) \subset \AAA$.
\end{enumerate}
\end{theorem}
Let $u^*$ be a steady state of the NSE.   In \cite{FJKrT2} the term {\it determining form}
was introduced for the following evolution equation
\begin{equation}\label{detform}
\frac{dv}{d\tau}=F(v)=-\|v-JW(v)\|_{X^0} ^2\ (v-Ju^*)\;, \quad v \in \BBB_{X}^{3R}(Ju^*)\subset X \;,
\end{equation}
which is an ordinary differential equation in the true sense, i.e. $F$ is globally Lipschitz in
$\BBB_{X}^{3R}(Ju^*)$, which is a positively invariant set for \eqref{detform} (here $\mu$, $h$ and
$W$ are determined by $\rho=4R$).
Moreover, by (ii),(iii) the
set of steady states of \eqref{detform} is precisely
$\{Ju(\cdot)$, where $u(\cdot) \subset \AAA \}$.  In fact, it was shown in \cite{FJKrT2} (and we will repeat this proof later)
that every solution of \eqref{detform} approaches
a steady state, as $\tau\to \infty$.
Thus, trajectories in the global attractor of the NSE
are readily recognized/realized/identified by the long-time behavior of the determining form.
Note that at each ``evolutionary time" instant $\tau$
we may write $v(\tau)=v(\tau,s)$, $-\infty < s < \infty$, just as for the NSE written in the form \eqref{NSE},
at each instant $s$ one may write $u(s)=u(s,x)$ to denote the dependence on the suppressed spatial variable $x$.

This type of determining form and its properties have also been established for
the subcritical surface quasigeostrophic equation (SQG), a damped, driven nonlinear Schr\"odinger equation (NLS), and a damped, driven Korteveg-de Vries  equation (KdV)
\cite{JMST,JST1,JST2}.  The analysis used to prove analogs of Theorem \ref{mainthm} for these systems differ from that for the NSE.
The approach for the SQG involves the Littlewood-Paley decomposition and Di Giorgi techniques in order to obtain $L^p$ estimates for \eqref{weqn} over the full subcritical range.
For the weakly dissipative NLS and KdV the analysis uses of certain compound functionals
resulting in different spaces $X$ and $Y$ in each case.  An earlier type of determining form in which trajectories in
the global attractor the 2D NSE were identified with traveling wave solutions $v(\tau,s)=v(0,\tau+s)$ was developed in \cite {FJKrT1}.  In this paper we focus on the dynamics of the determining form in \eqref{detform}, describing completely the basin of attraction for each steady state, and distinguishing between lower bounds on the rate of convergence toward $Ju^*$ and the rate toward any other trajectory in $\AAA$.

\section{One-dimensional ODE determining the global dynamics of the NSE}\label{sec3}

In this section we investigate the dynamical behavior of the determining form \eqref{detform}.  Remarkably, the global dynamics of the NSE, i.e., the trajectories on the global attractor of the NSE,
are precisely the liftings of the steady state solutions of a one-dimensional  ODE  \eqref{titi6}, below,
which will be referred to as the ``characteristic parametric
determining form".

The determining form \eqref{detform} is equivalent to
\begin{equation} \label{titi1}
\begin{split}
\frac{d}{d\tau}(v - Ju^*) & = -\| v - JW(v) \| _{X^0} ^2 (v - Ju^*) \;,\\
 v(0) &= v_0 \in \BBB_X ^{3R}(Ju^*) \;,
\end{split}
\end{equation}
whose solution can be written as
\begin{equation} \label{titi2}
v(\tau) - Ju^ * = \theta (\tau) (v_0 - Ju^*) ,
\end{equation}
where
\begin{equation} \label{titi3}
\begin{split}
\theta (\tau) &=  \exp \left( - \int _0 ^ \tau \| v(\sigma) - J W(v(\sigma)) \| _{X^0} ^ 2 \, d\sigma \right) \\
 \theta(0) &= 1, \text{ and } \theta(\tau) \in [0,1].
\end{split}
\end{equation}
Observe that
\begin{equation} \label{titi4}
\frac{d\theta}{d\tau} = - \theta \| v(\tau) - JW(v(\tau)) \| _{X^0} ^2 .
\end{equation}
Thanks to \eqref{titi2} we have
\begin{equation} \label{titi5}
v(\tau) = \theta(\tau) v_0 + (1-\theta (\tau)) Ju^*
\end{equation}
and combined with \eqref{titi4} we have a one-dimensional ODE for the characteristic determining parameter $\theta$:
\begin{equation} \label{titi6}
\begin{split}
\frac{d\theta}{d\tau} &=  -\theta \| \theta v_0 + (1-\theta) Ju^* - JW(\theta v_0 + (1-\theta)J(u^*)) \| _{X^0} ^2 =: \Phi(\theta;v_0,u^*) \\
 \theta(0) &= 1.
\end{split}
\end{equation}
We term equation \eqref{titi6} the characteristic parametric determining form.

Observe that \eqref{titi5} states that the solution of the determining form \eqref{detform} is a convex combination of the given initial data $v_0$ and the steady state projection $Ju^*$, where  the convexity parameter is the characteristic determining parameter. 
Moreover, the evolution of \eqref{detform} is equivalent to the evolution of the one parameter ODE \eqref{titi6}, the characteristic parametric determining form.  Note also that \eqref{titi5} gives a separation of variables, i.e.
\begin{equation*}
\begin{split}
v(\tau,s) &= \theta(\tau)v_0(s) - (1-\theta(\tau)) Ju^* \\
\text{for} \quad \tau\ &\geq 0 , ~ s\in \RR, \text{ where } v_0 \in \BBB_X ^{3R}(Ju^*).
\end{split}
\end{equation*}
\subsection{The behavior of the $\theta(\tau)$: Lower bounds on convergence rate}

\begin{theorem}\label{titithm}
The limit $\lim _{t\rightarrow \infty} \theta(\tau) = \bar{\theta}$ exists and
\begin{enumerate}
\item $\bar{\theta} = 0$ if and only if
\[ \int _0 ^ \infty \| v(\sigma ) - J W(v(\sigma)) \| _{X^0} ^ 2 \, d\sigma = \infty \quad \text{and} \quad
v(\tau) \to \bar{v} = Ju^*\;, \text{ as } \tau \to \infty\;.\]
\item $\bar{\theta} \in (0,1]$ if and only if
\[ \int _0 ^ \infty \| v(\sigma ) - J W(v(\sigma)) \| _{X^0} ^ 2 \, d\sigma < \infty\;. \]
Moreover, in this case $\bar v=\bar\theta v_0+(1-\bar \theta)Ju^*$ satisfies $W(\bar v)(\cdot) \subset \cA$.
\item $\Phi(\bar\theta;v_0,u^*)=0$ if and only if $\bar\theta$ satisfies (i) or (ii); where
$\Phi(\bar\theta;v_0,u^*)$ is given in \eqref{titi6}.  The zeros, $\bar \theta$, of
$\Phi(\theta;v_0,u^*)$ are called characteristic determining values.
\item If $\bar{\theta} = 0$ then
\begin{equation}\label{thalf}
\theta(\tau) \geq \frac{1}{(1+2\tc^2\| v_0 - Ju^*\|_{X} ^2 \tau) ^{1/2}}\;, \quad \tau\geq 0\;,
\end{equation}
\begin{equation}\label{thalfX}
\| v(\tau) - Ju^* \| _X \geq \frac{\|v_0 - Ju^* \|_{X}}{(1+2\tc^2\|v_0 - Ju^* \| ^2 _{X} \tau ) ^{1/2}}\;, \quad \tau \geq 0
\end{equation}
and
\begin{equation}\label{thalfXX}
\| v(\tau) - Ju^* \| _{X^0} \geq \frac{\|v_0 - Ju^* \|_{X^0}}{(1+2\tc^2\|v_0 - Ju^* \| ^2 _{X} \tau ) ^{1/2}}\;, \quad \tau \geq 0.
\end{equation}
where $\tc$ is a constant depending only on $c_J$ and the Lipschitz constant of $W$.

\item If $\bar{\theta} \in (0,1]$ then $\bar{v} = v_0 \bar{\theta} + (1-\bar{\theta} )Ju^*$ and
\begin{equation} \label{whole}
\theta(\tau) - \bar{\theta} \geq \frac{1-\bar{\theta}}{1+\tc^2\|v_0 - Ju^* \| ^2 _{X} \tau }\;,
\quad \tau \geq 0,
\end{equation}
\begin{equation}\label{wholeX}
\| v(\tau) - \bar{v} \| _X \geq \frac{(1-\bar{\theta})\|v_0 - Ju^*\|_{X}}{1+\tc^2\|v_0 - Ju^* \| ^2 _{X} \tau }\;, \quad \tau \geq 0
\end{equation}
and
\begin{equation}\label{wholeXX}
\| v(\tau) - \bar{v} \| _{X^0} \geq \frac{(1-\bar{\theta})\|v_0 - Ju^*\|_{X^0}}{1+\tc^2\|v_0 - Ju^* \| ^2 _{X} \tau }\;, \quad \tau \geq 0.
\end{equation}
\end{enumerate}
\end{theorem}

As a consequence of the above theorem we have the following corollary which states that the
trajectories in the global attractor of the NSE are determined in a unique fashion by the characteristic determining values, i.e., zeros of a
real-valued function of a real variable, the vector field of the parametric determining form, $\Phi(\theta;v_0,u^*)$.

\begin{corollary} Let $v_0 \in \BBB^{3R}_X(Ju^*)$, and let $\bar\theta$ satisfy 
$\Phi(\bar\theta;v_0,u^*)=0$.  \hfill\break Then $w(\cdot)=W(v_0 \bar{\theta} + (1-\bar{\theta} )Ju^*)(\cdot) \subset \cA$.  Conversely, for every $u(\cdot) \subset \cA$, there exists $v_0 \in \BBB_X^{3R}(Ju^*)$ and $\bar\theta \in [0,1]$ such that $\Phi(\bar\theta, v_0,u^*)=0$ and \hfill\break
$u(\cdot)=W(v_0 \bar{\theta} + (1-\bar{\theta} )Ju^*)(\cdot)$.
\end{corollary}

The proof of the first part of the corollary follows immediately from part (iii) of Theorem \ref{titithm} and part (iii) of Theorem \ref{mainthm}.  To prove the second part of the corollary
we take $v_0=Ju$ where $u(\cdot) \subset \AAA$ and $\bar\theta=1$, because in this case $v_0$ is a steady state of the determining form \eqref{titi1}.

\begin{remark} In fact for a given $u(\cdot)\subset \AAA$ there might be infinitely many
$v_0$ satisfying the second part of the proposition.  They must have the form
$$
v_0(\cdot)=\alpha Ju(\cdot) + (1-\alpha)Ju^* \quad \text{for some} \quad
\alpha \ge 1\;, \quad \text{with} \quad v_0 \in \BBB_X^{3R}(Ju^*)\;.
$$
In the above proof
we have shown the existence of at least one such $v_0$, which is  the trivial choice.
However, in the computational section, below, we take a nontrivial choice for $v_0$ to demonstrate the validity of the analytical results.
\end{remark}

Next we give the proof of Theorem \ref{titithm}.
\begin{proof}
The existence of the limit $\bar\theta$ follows from the fact that $\theta(t) \in [0,1]$ is monotonic and non-increasing.
Part (i) is immediate.

To prove part (ii) observe that since $\bar\theta > 0$, we have
\begin{align}\label{finiteint}
\int_0^ \infty \| v(\sigma)-JW(v(\sigma))\|_{X^0}^2 \ d\sigma < \infty \;.
\end{align}
 As a result
of \eqref{titi5} we conclude that
$$
v(\tau)=\theta(\tau)v_0 + (1-\theta(\tau))Ju^* \to \bar v=\bar\theta v_0 + (1-\bar\theta)Ju^*\;,
$$
as $\tau \to \infty$.
Since $J$ and $W$ are continuous maps we have $JW(v(\tau)) \to JW(\bar v)$, as $\tau \to \infty$.
Thus, the integrand in \eqref{finiteint}, $\| v(\sigma)-JW(v(\sigma))\|_{X^0}^2$, converges to $\| \bar v(\sigma)-JW(\bar v(\sigma))\|_{X^0}^2$, as
$\sigma \to \infty$.  Since the integral is finite we have
$$
\| v(\sigma)-JW(v(\sigma))\|_{X^0}^2 \to  0\;,
\quad \text{as} \quad \sigma \to \infty\;.
$$
As a consequence, we have $\bar v-JW(\bar v)= 0$.  Thanks to part (iii)
of Theorem \ref{mainthm}, $W(\bar v)(\cdot) \subset \cA$.

To prove (iv) observe that $Ju^* = JW(Ju^*)$, therefore
\begin{equation} \label{bigguy}
\begin{aligned}
\| v - JW(v) \| _{X^0} &= \| v - Ju^*  + JW(Ju^*) - JW(v) \| _{X^0} \\
& \le \|v-Ju^*\|_{X^0}+ \|JW(Ju^*)-JW(v)\|_{X^0} \\
(J:\dot H^2 \to \dot H^1 \text { bounded}) \qquad & \le \|v-Ju^*\|_{X^0}+c_J c \sup_{s \in \bR}|A(W(Ju^*)-W(v)| \\
& \le \|v-Ju^*\|_{X^0}+ c_J c \|W(Ju^*)-W(v)\|_Y \\
(W: X \to Y \text{ Lipschitz}) \qquad & \le \|v-Ju^*\|_{X^0}+ c_J c L_{\text{Lip}}\|Ju^*-v\|_X \\
(\|\cdot\|_X \text{ stronger than } \|\cdot \|_{X^0}) \qquad &\leq \tc \| v - Ju^* \|_{X} \\
\text{by \eqref{titi5}}\qquad & =\tc\theta\|v_0-Ju^*\|_{X}
\end{aligned}
\end{equation}
As a consequence of \eqref{bigguy}  and \eqref{titi4} we have
\[
\frac{d\theta}{d\tau} \geq -\theta ^3 \tc^2\| v_0 - Ju^* \| _{X} ^2 ; \quad \theta(0) = 1.
\]
We integrate to obtain \eqref{thalf}, then use \eqref{titi5} to get \eqref{thalfX} and \eqref{thalfXX}.

To prove (v) suppose $\bar{\theta}\in (0,1]$, and $\bar{v} = \bar{\theta}v_0 + (1-\bar{\theta })Ju^*$. Note that by part (ii) $\bar{v}=JW(\bar{v})$, so that $\| \bar{v} - JW(\bar{v}) \| _{X^0} = 0$. In view of this and $\eqref{titi4}$, we have
\begin{align*}
\frac{d}{d\tau}(\theta - \bar{\theta}) &= -\theta \| v - JW(v) \| _{X^0} ^2 \\
&= -\theta \left[ \| v - JW(v) \| _{X^0} - \| \bar{v} - JW(\bar{v} ) \| _{X^0} \right] ^2\;,
\end{align*}
while proceeding as in \eqref{bigguy}, we have
\begin{align*}
\left \vert \| v - JW(v) \| _{X^0} - \| \bar{v} - JW(\bar{v}) \| _{X^0} \right \vert &\leq \| (v - \bar{v}) - (JW(v) - JW(\bar{v}) ) \| _{X^0} \\
&\leq \tc\| v - \bar{v} \| _{X} \\
&\leq \tc (\theta - \bar{\theta} ) \| v_0 - Ju^* \| _{X}  \;.
\end{align*}
Since $\theta(\tau) \geq \bar{\theta}$ and $\theta(\tau) \in (0,1]$,
\begin{align*}
\frac{d}{d\tau} (\theta-\bar{\theta}) &\geq -\theta (\theta - \bar{\theta} ) ^2 \tc^2 \| v_0 -Ju^*  \| _{X} ^2  \\
&\geq -(\theta - \bar{\theta})^2 \tc^2\| v_0 - Ju^* \| _{X} ^2\;,
\end{align*}
with $\theta(0) = 1$. Once again, integrating yields \eqref{whole}, and applying \eqref{titi5}
gives \eqref{wholeX} and  \eqref{wholeXX}.
\end{proof}

\section{Accelerating the convergence rate} \label{sec4}
Following the proof of Theorem \ref{mainthm}  in \cite{FJKrT2} one can modify
the determining form slightly and prove similar statements for
the ODE (a modified determining form)
\begin{equation}\label{detform3}
\frac{dv}{d\tau}=-\|v-JW(v)\|_{X^0}\ (v-Ju^*)\;, \quad v(0)=v_0 \;.
\end{equation}
The solution to \eqref{detform3} can be expressed through the convexity parameter
$$
\tilde\theta(\tau)=\exp\left(\int_0^\tau -\|v(s)-JW(v(s))\|_{X^0} \ ds \right) $$
satisfying
$$
\frac{d \tilde\theta}{d\tau}=-\tilde\theta \|v(s)-JW(v(s))\|_{X^0} \;.
$$
Due to the reduced power on the norm, one can then follow the proof of Theorem \ref{titithm}, to derive faster, but still algebraic, lower bounds on the rates of convergence in analogs of \eqref{thalf}-\eqref{wholeX}.

For yet another parametric ODE
\begin{align} \label{etaeqn}
\frac{d\eta}{d\tau} = -\|v(\tau)-JW(v(\tau))\|_{X^0}\;, \quad  \eta(0)=1\;,
\end{align}
where $v(\tau)=\eta(\tau)v_0+(1-\eta(\tau))Ju^*$, one can again follow the proof of Theorem \ref{titithm} to show an exponential lower bound on the rate of convergence to $\bar \eta$
\begin{align}\label{exprate}
\eta(\tau)-\bar \eta \ge (1-\bar \eta) e^{-c\tau}
\end{align}
for both cases $\bar \eta=0$ and $\bar \eta \in (0,1)$.
The ``determining form" in the space of trajectories associated with \eqref{etaeqn} is
\begin{equation}\label{detform4}
\frac{dv}{d\tau}=-\|v-JW(v)\|_{X^0}\ (v_0-Ju^*)\;, \quad v(0)=v_0 \;,
\end{equation}
which is, strictly speaking, not an ODE, due to the involvement of the initial condition $v_0$ in the vector field. The solution to \eqref{detform4}, however, follows the same trajectory as that for
\eqref{detform3} and \eqref{detform}, only in each case the parametrization is different.

\section{Further properties of solutions of the determining form} \label{sec5}
\begin{proposition}\label{shifty}
Suppose $\tilde{v} \in \BBB _X ^{3R} (Ju^*)$. Denote by $\tilde{v}_\sigma (s) = \tilde{v}(s+\sigma)$, then
\begin{align}\label{shifteqn}
 W(\tilde{v} _\sigma) (s) = W(\tilde{v})(\sigma + s)\;.
\end{align}
\end{proposition}
\begin{proof}
Let $w=W(\tilde{v})$ be the unique bounded solution of the equation
\begin{align}\label{E1}
\frac{d}{ds} w + \nu A w + B(w,w) = f - \mu (Jw-\tilde{v}) \;.
\end{align}
Notice that $\|\tilde{v}_\sigma\|_X=\|\tilde{v}\|_X < 3R$.
Since \eqref{E1} is autonomous it is clear that if $w(s)$ corresponds to $\tilde{v}(s)$ then $w(s+\sigma)$ is the unique bounded solution of \eqref{weqn} corresponding to $v(s)=\tilde{v} (s+\sigma) = \tilde{v}_\sigma (s)$, and so \eqref{shifteqn} follows.
\end{proof}
\begin{corollary}\label{noway}
If $\tilde{v} \in \BBB _X ^\rho (Ju^*)$ is periodic with respect to $s$ with period $P >  0$, then $W(\tilde{v})$ is also periodic with period $P$.  In particular, if $\tilde v$ is independent of $s$, then $w=W(\tilde v)$ is the
unique steady state of \eqref{E1}.
\end{corollary}
\begin{proof}
By the assumption we have $\tilde{v}(s+P) = \tilde{v}_P (s) = \tilde{v}(s)$. That is $\tilde{v}_P = \tilde{v}$. Hence, by Proposition \ref{shifty} $W(\tilde{v}_P)(s) = W(\tilde{v})(s+P)=W(\tilde{v})(s)$. Thus, $W(\tilde{v})$ is periodic with period $P$.  If $\tilde v (s)=\tilde v ^*=\text{constant}$, then $\tilde v(s+P)=\tilde v^*$ for every
$P \in \RR$.  Thus the unique bounded solution of \eqref{E1} satisfies $w(s+P)=w(s)$
for  all $s \in \RR$, and all $P \in \RR$, which implies $w=\text{constant}$ and as a result $W(\tilde v)$ is
a steady state of \eqref{E1}.
\end{proof}

\begin{remark} For an alternative proof of Corollary \ref{noway} in the case where
$\tilde{v} ^*$ is independent of $s$, consider the steady state version of  \eqref{weqn}.
As in the case of the NSE one can show that 
\eqref{weqn} with $\tilde{v}^*$ in the right-hand side has a steady state solution.
Since it is bounded with respect to $s$, it is the only solution (one can also show the uniqueness
directly).  Thus the steady state solution $w^*$ is the only bounded solution and hence $w^*=W(\tilde{v}^*)$.
\end{remark}

\begin{proposition}
Let $P^* >0$, and assume that there is no solution of the NSE (or the underlying equation), which is periodic with period $P^*$. Then all the periodic initial data $v_0\in \BBB _X ^{3R}(Ju^*)$ with period $P^*$ converge to $Ju^*$. Moreover, if $v_0$ is independent of $s$, then $v(t,s)$ is also independent of $s$, for all $t >0$, and the image of its limit $W(\bar v)$, as $t \to \infty$, is a steady state of the NSE.
\end{proposition}
\begin{proof}
Thanks to \eqref{titi5} we have $v(\tau) = \theta(\tau) v_0 + (1-\theta(\tau))Ju^*$. Since $v_0(s+P^*) = v_0(s)$ for all $s\in \RR$, then if $\theta(\tau) \rightarrow \bar{\theta} >0$, as $\tau\rightarrow \infty$, we have
$\bar v=\bar \theta v_0+(1-\bar \theta)Ju^*$ satisfying
$\bar{v} (s + P^*) = \bar{v}(s)$, where $W(\bar v)$ is a trajectory in $\AAA$, which is periodic with period $P^*$, a contradiction. Hence $\bar{\theta} = 0$.  The statement regarding initial data which are independent of $s$ follows from \eqref{titi2}.
\end{proof}
\begin{proposition}
The determining form \eqref{detform} does not have any traveling wave solutions.
\end{proposition}
\begin{proof}
This follows from the fact that each solution of  \eqref{detform} converges to a steady state of \eqref{detform}. 
\end{proof}
\section{Computational evidence that rate estimates are achieved} \label{sec6}
Here we carry out simulations of \eqref{titi6}, the characteristic parametric determining form, to demonstrate that the lower bounds on the rates of convergence
in Theorem \ref{titithm} are achieved for a particular forcing term of the NSE \eqref{NSE}. The use of either \eqref{detform} or \eqref{titi6} for the expressed purpose of locating the attractor is a subject of future work.

Note that the only connection \eqref{titi6} has to the NSE is
through the map $W$. While the rigorous construction of $W$ in \cite{FJKrT2} involves solving \eqref{weqn} with initial data
$w(s_0)=0$, and taking $s_0\to -\infty$, we can effectively approximate $w(s)$ for $0 < s_1 \le s \le s_2$ by taking
$w(0)=0$ and solving forward in time, i.e., taking $s_1$ sufficiently large.  The final time $s_2$ is chosen so as to compute the sup norm of what will be a periodic function of $s$. Just as sufficiently large $\mu$, and correspondingly small $h$ guarantee
that the map $W$ is well-defined, so for similarly chosen $\mu$, $h$ will  $|w(s)-u(s)| \to 0$, as $s\to \infty$, at an exponential rate, if we take $v=Ju$ as is done in data assimilation \cite{AOT}
(see also the computational study \cite{GOT}).  Using Theorem \ref{mainthm} (iii),  we empirically determine that $s_1=1.0$ is a sufficient relaxation time by recovering a particular periodic solution $\tilde{u}(\cdot) \subset \AAA$ (see Figure 1).  The interpolant $J$ is the projection onto the Fourier modes $(i,j)$, $|i|,|j| \le 5$, and the relaxation parameter $\mu=150$.

We consider two cases for $v_0$: one where we know $\theta \to \bar\theta > 0$,
and another where we can expect $\theta \to 0$. For the former we take for $v_0= J(2\tilde{u}-u^*)$, a point on the ray from
$Ju^*$ through $J\tilde{u}(s_1)$, beyond $J\tilde{u}(s_1)$ (double the distance).  In the latter case, $v_0=J\tilde{u}+\delta$, where
$$
\hat{\delta}_{1,1}=  \hat{\delta}_{-i,-j} =0.5 \;,\quad  
\text{and} \quad  \hat{\delta}_{i,j} =0 \quad \text{for} \quad i,j \neq \pm(1,1)\;.
$$

Several techniques are used for computational efficiency.  The NSE and \eqref{weqn} are
solved in vorticity form.  The steady state $u^*$ and force $f$ are defined through the curl;
$$
\omega^*=\nabla \times u^*, \ \phi=\nabla \times f, \text{ where }
\hat{\omega}^*_{3,4}=24+36i, \ \hat{\omega}^*_{5,0}=60+84i, 
\text{ and } \hat{\phi}_{i,j}=25 \hat{\omega}^*_{i,j}\;,
$$
with complex conjugate values at wave vectors $(-3,-4)$ and $(-5,0)$.
 Otherwise all Fourier coefficients are 0.
Rather than compute
the supremum norm at each step in $\tau$, this norm is sampled
at 150 values of $\theta \in [0,1]$ (see Figure \ref{fig2}).  This approach avoids solving over an extended interval
in $s$ which would be needed in a direct approach due to compounding relaxation times.
For each sample point $\theta$,  the solution $w(s)$, $s\in [1.0,1.5]$ is computed by solving (in curl form)
a fully dealiased, $256\times256$ mode pseudospectral discretization of \eqref{weqn} with $\nu=\k0=1$, and $v=\theta v_0 + (1-\theta)J(u^*)$. The time-stepper to solve \eqref{weqn} is a third order Adams-Bashforth scheme, with $\Delta s=5\times 10^{-5}$,
which respects a CFL condition (see \cite{olson2008determining} for more details). The supremum norm in \eqref{titi6} is concurrently computed along with $w$.
Using this approximation of the supremum norm dependence on $\theta$ in \eqref{titi6} results in a piecewise linear
ODE which is solved exactly over each subinterval in $\tau$.
Higher resolution sampling for $\theta \sim 10^{-4}$ is needed to produce smooth convergence for large $\tau$ in Figure \ref{fig3}, demonstrating that the lower bounds
on the rates in Theorem \ref{titithm} are achieved.

We follow a similar procedure for \eqref{etaeqn} to make Figure \ref{fig4}
which indicates that $\eta$ converges at an exponential rate, regardless of whether
$v \to Ju^*$ or $v \to J\tilde{u}$, though in the case of the former the rate still appears to be slower.  We also demonstrate in Table \ref{secant} the advantage of having a real-valued function of
a real variable, i.e., the function $\Phi$ in \eqref{titi6}, identify trajectories in $\AAA$ by applying the secant method to converge
in $\tilde{u}$ in just 8 steps.

\begin{figure}[h]
  \centerline{\includegraphics[scale=.47]{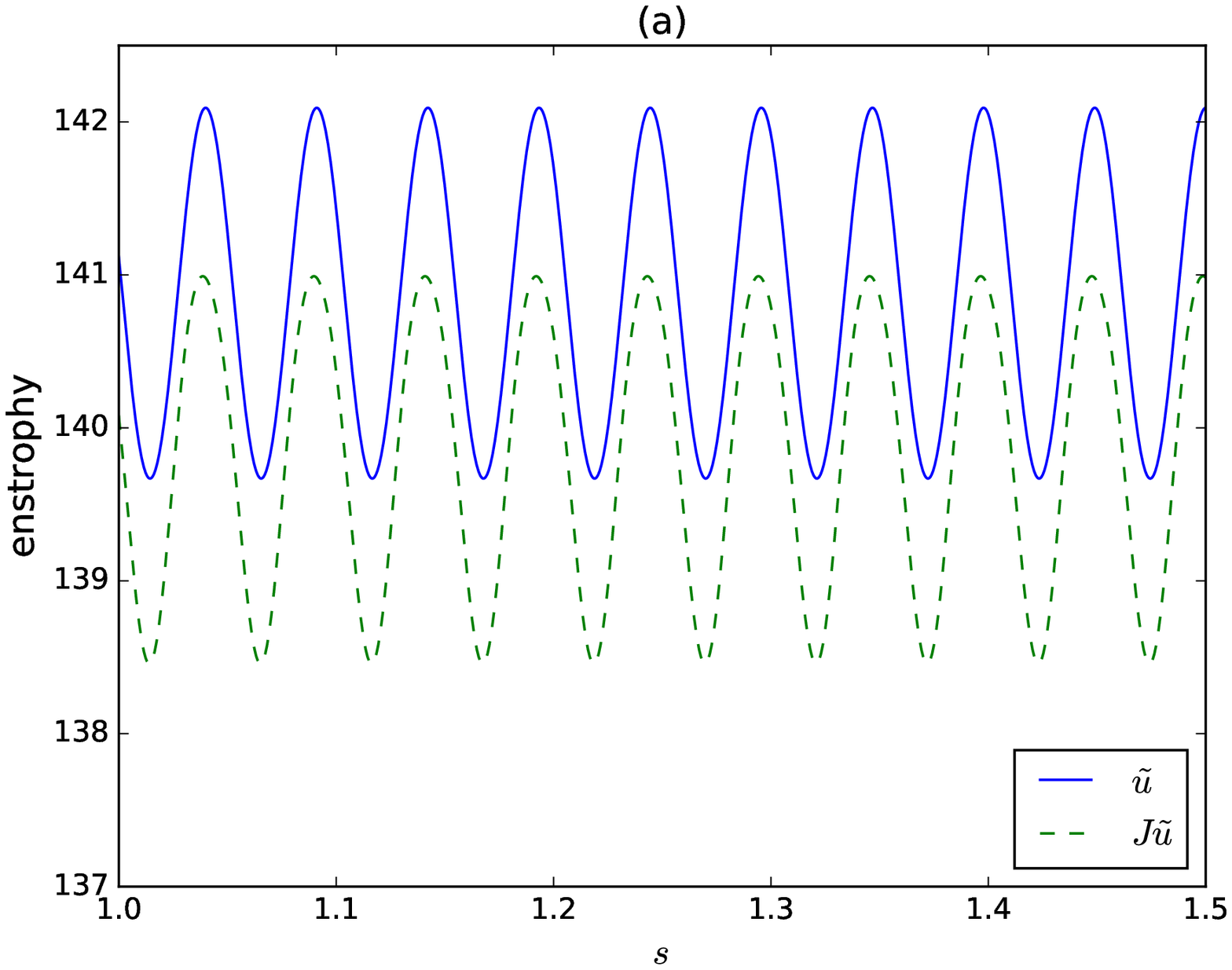}
  \includegraphics[scale=.47]{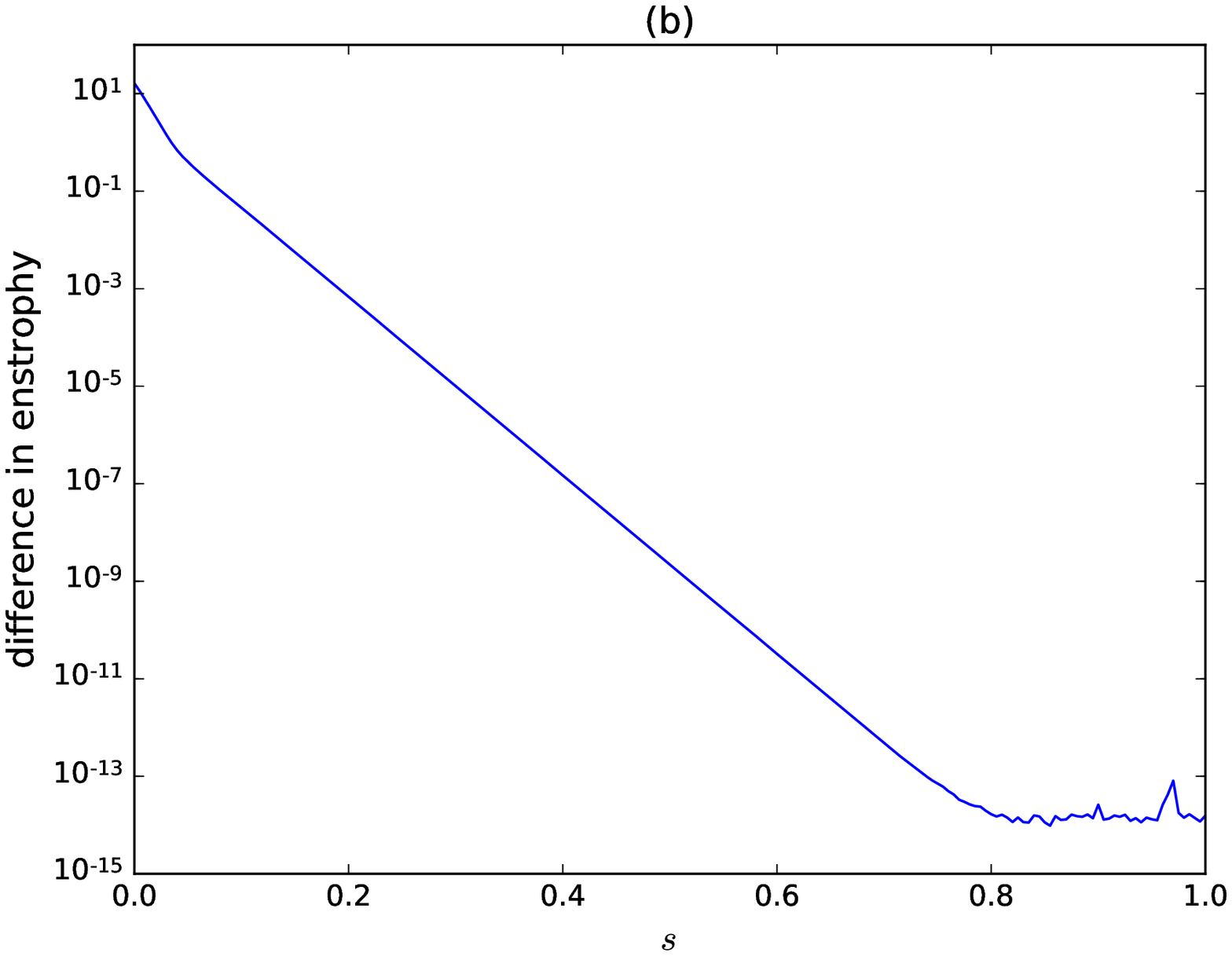}}
\caption{(a) $\tilde{u}(\cdot)\subset \AAA$.
\quad (b) Relaxation of solution of \eqref{weqn} to $\tilde{u}(\cdot)$.}
\label{fig1}
\end{figure}
\begin{figure}[h]
  \centerline{\includegraphics[scale=.25]{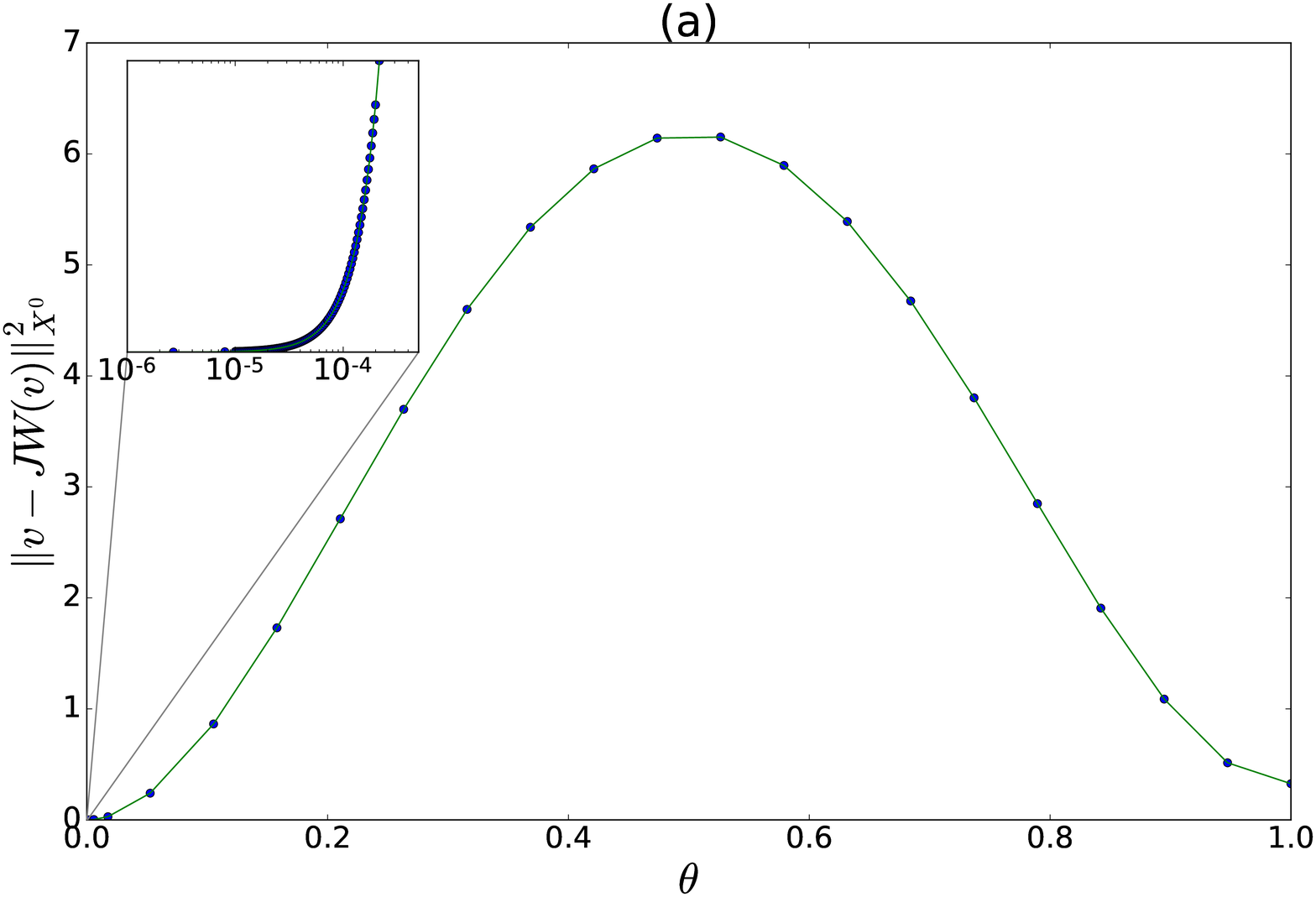}
   \includegraphics[scale=.25]{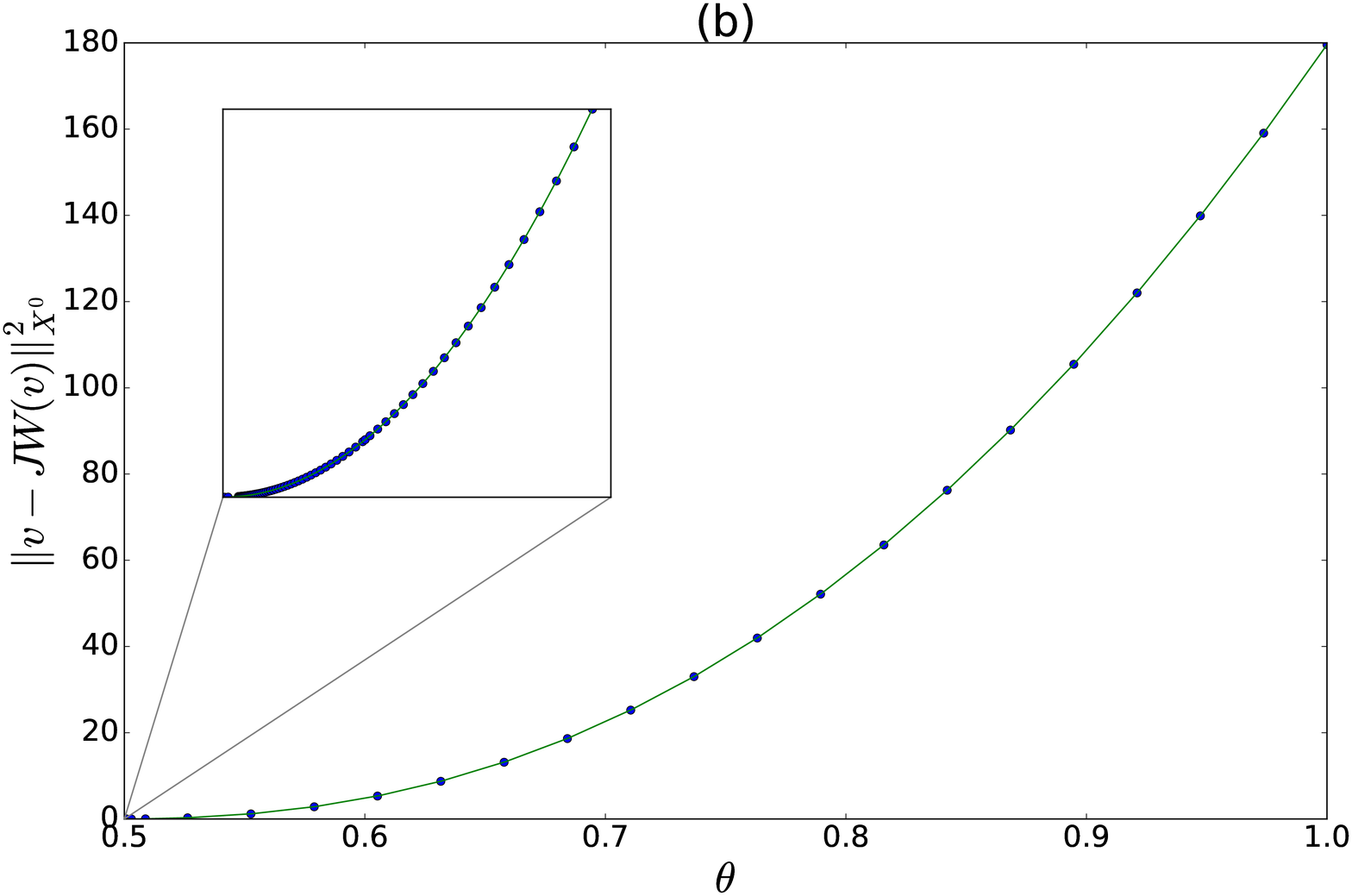}}
\caption{Sampling of supremum norm: (a)  $v_0=J\tilde{u}+\delta $, where we expect $\theta \to 0$,
(b)  $v_0= J(2\tilde{u}-u^*)$, where we know $\theta \to \bar\theta =.5$}
\label{fig2}
\end{figure}
\begin{figure}[h]
 \psfrag{g}{a}
  \centerline{\includegraphics[scale=.47]{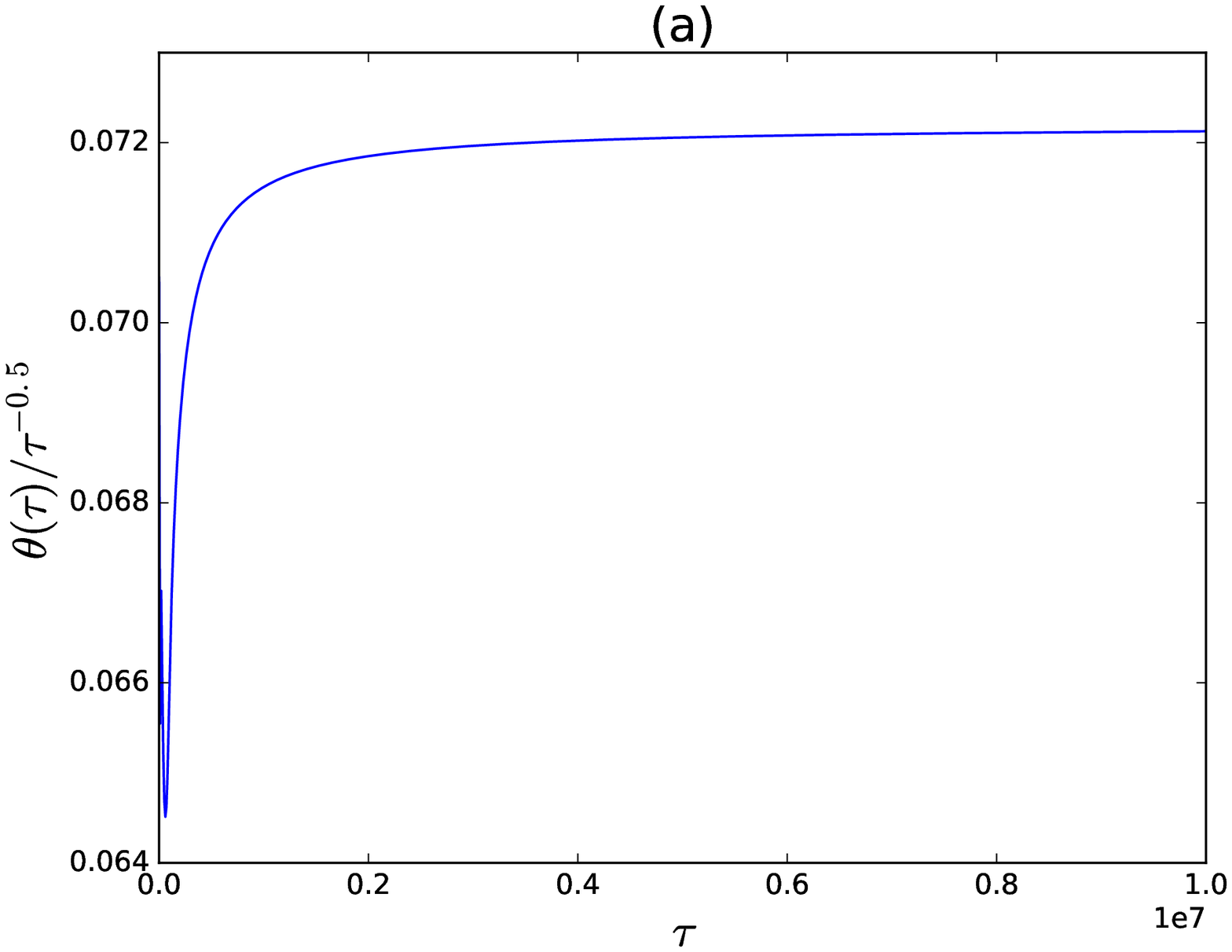}
  \includegraphics[scale=.47]{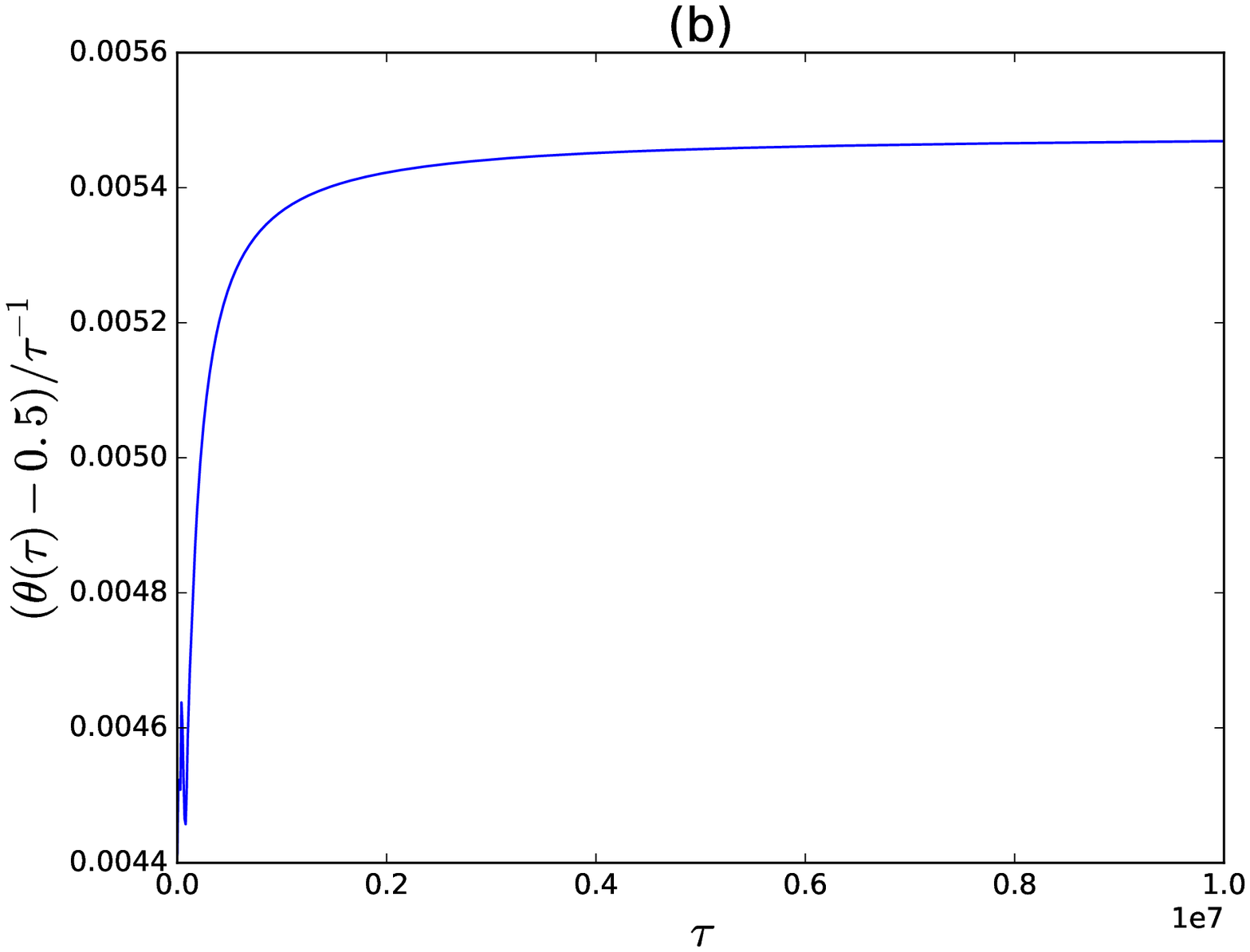}}
\caption{Convergence for \eqref{titi4}: (a) $v_0=J\tilde{u}+\delta$, (b) $v_0=J(2\tilde{u}-u^*)$}
\label{fig3}
\end{figure}

\begin{figure}[h]
 \psfrag{g}{a}
  \centerline{\includegraphics[scale=.65]{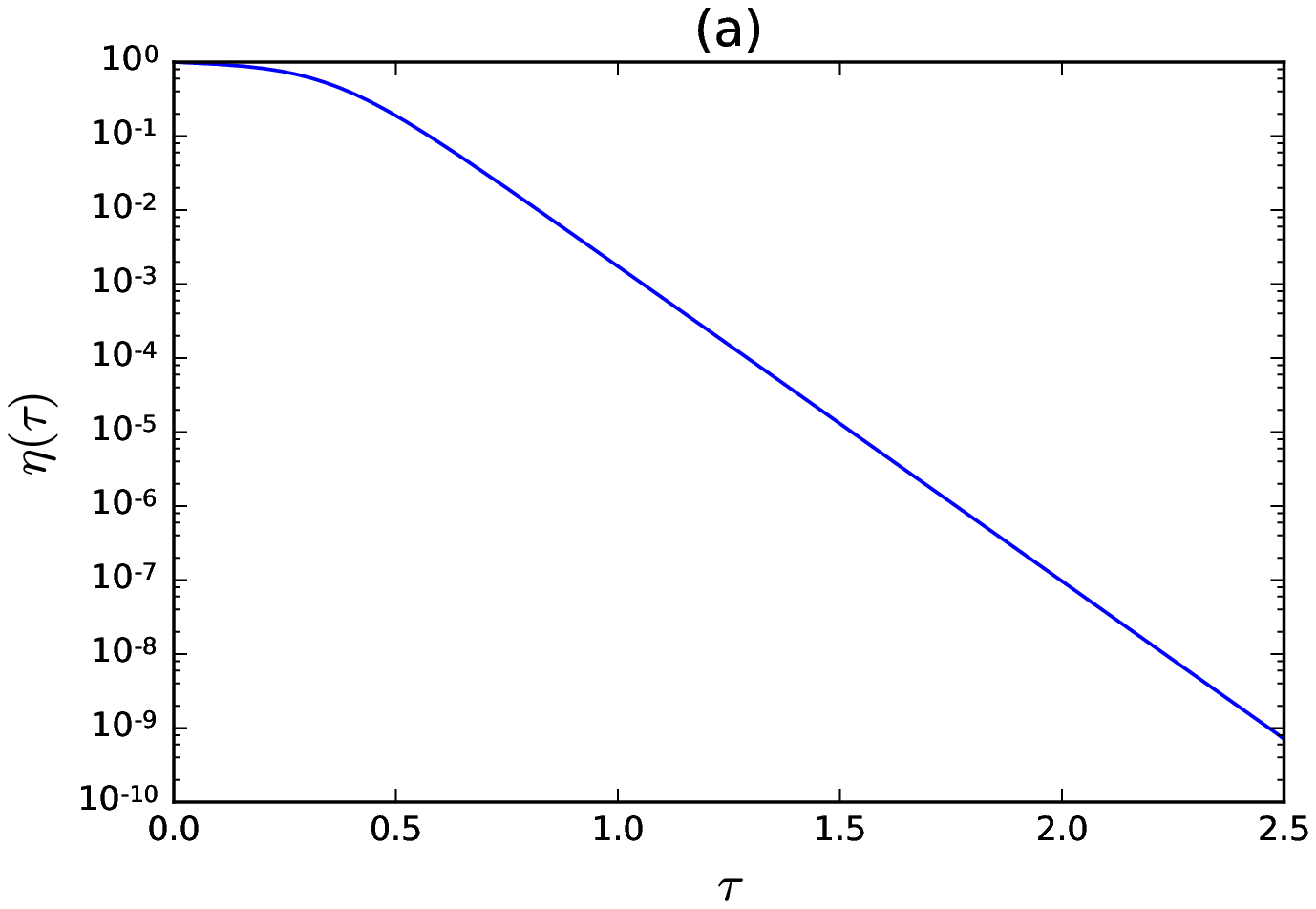}
  \includegraphics[scale=.65]{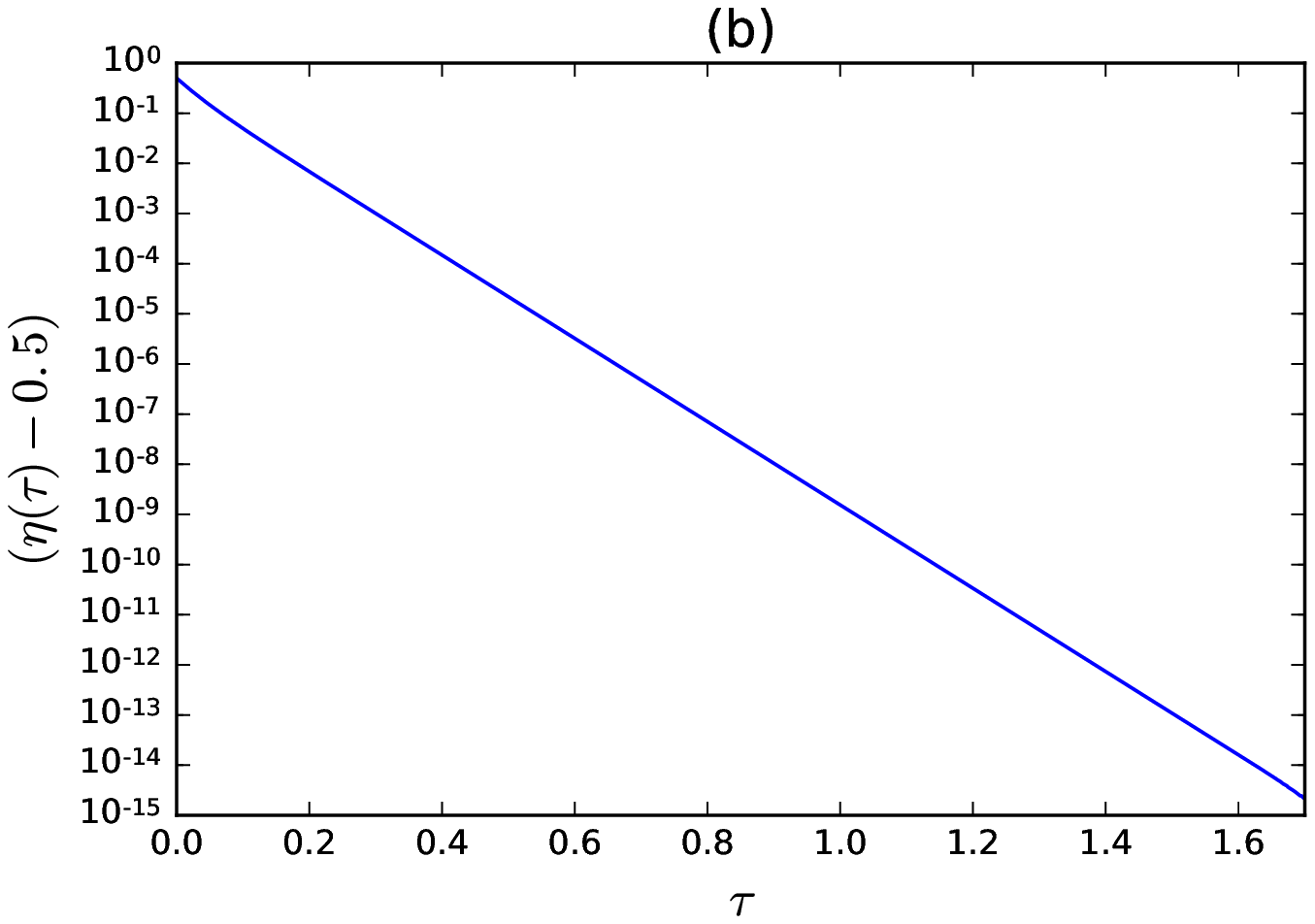}}
\caption{Convergence for \eqref{etaeqn}: (a) $v_0=J\tilde{u}+\delta$, (b) $v_0=J(2\tilde{u}-u^*)$}
\label{fig4}
\end{figure}
\begin{table}
\begin{center}
\begin{tabular}{|c|c||c|c|c|c|} \hline
$i$&  $\eta_i$ &  $\|v_i-JW(v_i)\|_{X^0}$  \\ \hline \hline
0  &	1.00000000000000	& 13.4013441411815 \\ \hline
1 &	0.980000000000000 &	12.8017448867074 \\ \hline
2 &	0.552989966509007 &	1.09311945051355 \\ \hline
3 &	0.513124230285987 &	0.255930982604035 \\ \hline
4  &	0.500937157201634 &	1.792679690457618E-002 \\ \hline
5  &	0.500019210387369 &	3.669263629982810E-004 \\ \hline
6  &	0.500000029216005 &	5.580203907596127E-007 \\ \hline
7 &	0.500000000000913 &	1.745075917466212E-011 \\ \hline
8 & 	0.500000000000000 &	2.455150177114436E-014 \\ \hline
\end{tabular}
\end{center}
\caption{Secant method, $v_0=J(2\tilde{u}-u^*)$, $v_i=\eta_i v_0 + (1-
\eta_i)Ju^*$}
\label{secant}
\end{table}

\section{Final remarks} \label{sec7}
The steady state solutions $\bar \theta$ of \eqref{titi6}, i.e., characteristic determining values which are the zeros of the function
$\Phi$ in \eqref{titi6},  capture all the trajectories on the global
attractor $\AAA$ of the NSE through $W(\bar \theta v_0+(1-\bar \theta )Ju^*)$ as $v_0$ varies
throughout $\BBB^{3R}_X(Ju^*)$.  We do not know if $W$ is differentiable
so using the Newton-Raphson method to find the zeros of $\Phi(\theta;v_0,u^*)$ may not work.
However, since $W$ is Lipschitz one can (and we do here) find the zeros
by the secant method.  In particular, every steady state of the NSE is realized as
the image $W(\bar \theta v_0+(1-\bar \theta )Ju^*)$, where $v_0$ is independent of $s$.  It is thus possible, and perhaps beneficial to study bifurcations of steady states through \eqref{titi6}.  Moreover, if $v_0$ is periodic (with positive minimal period), then
$W(\bar \theta v_0+(1-\bar \theta )Ju^*)$ is either periodic (with the same minimal period) or the steady state $u^*$.  If, on the other hand, $v_0$ is independent of $s \in \bR$,
then $W(\bar \theta v_0+(1-\bar \theta )Ju^*)$ is some steady state of the 2D NSE.  Rigorous lower bounds suggest the rate of convergence toward $Ju^*$ is slower than toward any other steady
state of the determining form.  It remains to study the complete basin of attraction for the determining form of the exceptional state
$Ju^*$.

\section*{Acknowledgements}
The work of C.F. was partially supported by National Science Foundation (NSF) grant
DMS-1109784 and Office of Naval Research  (ONR) grant  N00014-15-1-2333, that of M.S.J. by NSF grant DMS-1418911 and the Leverhulme Trust grant VP1-2015-036,  and D.L. by NSF grant DMS-1418911.  The work of E.S.T. was supported in part by the ONR grant N00014-15-1-2333 and the NSF grants DMS-1109640 and DMS-1109645.

\bibliographystyle{plain}
\bibliography{FJLT.bib}

\end{document}